\documentclass[english,a4paper,reqno,11pt]{amsart}
\usepackage[T1]{fontenc}
\usepackage{a4wide}
\usepackage{
    babel,
    lmodern,
    mathtools,
    amssymb,
    aligned-overset,
    enumerate,
    amsrefs,
    amsthm,
    thmtools,
    xcolor,
    hyperref,
    cleveref
}
\allowdisplaybreaks
\usepackage[all]{xy}

\newtheorem{teo}{Theorem}[section]

\newtheorem{lema}[teo]{Lemma}
\newtheorem{cor}[teo]{Corollary}
\newtheorem{prop}[teo]{Proposition}

\newcommand{\G}{\mathcal{G}}

\newcommand{\GG}{\mathcal{G}_2}

\def \si {\sigma}

\def \bet {\beta}
\def \tet {\theta}
\def \ga {\gamma}
\def \sii {\overline{\si}}
\def \gaa {\overline{\ga}}
\def \R {R^{\bet}}
\def \cbet {C(R)^{\bet}}
\def \Rh {R^{\bet_{\mathcal{H}}}}

\def \seta {\longrightarrow}

\def \HH {\mathcal{H}}
\def \L {\mathcal{L}}

\let\originaleqnarray\eqnarray  
\def\eqnarray{\settowidth{\arraycolsep}{$\mskip 0.5\thickmuskip$}\originaleqnarray}
\expandafter\let\expandafter\eqnarraystar
\csname eqnarray*\endcsname  
\expandafter\def\csname eqnarray*\endcsname
{\settowidth{\arraycolsep}{$\mskip 0.5\thickmuskip$}\eqnarraystar}

\begin{document}\thispagestyle{empty}

\title{Injectivity of the Galois map}
\author[Pedrotti, Tamusiunas]{Juliana Pedrotti, Thaísa Tamusiunas}
\address{Instituto de Matem\'atica, Universidade Federal do Rio Grande do Sul,
91509-900, Porto Alegre, RS, Brazil}
\email{julianabpedrotti@gmail.com}
\email{thaisa.tamusiunas@gmail.com}

\maketitle


\begin{abstract}
Given a Galois extension $R^{\beta} \subset R$, where $\beta$ is an action of a finite groupoid on a noncommutative ring, we present some conditions to the Galois map be injective.
\end{abstract}

\



\section{Introduction}
A Galois theory for groupoids acting on commutative rings was developed in \cite{bagio}, \cite{BST}, \cite{corta}, \cite{pata} and \cite{pataII}. Among other results, it was proven that, given an action $\beta$ of a groupoid $\G$ over a commutative ring $R$, there is a bijective correspondence between the wide subgroupoids of $\G$ and the $R^{\beta}$-subalgebras of $R$ that are $\beta$-strong and $R^{\beta}$-separable, where $R^{\beta}$ is the subalgebra of invariants of $R$ \cite[Theorem 4.6]{pataII}. This correspondence is called Galois correspondence, and the map which gives such correspondence is called Galois map. Furthermore, a Galois correspondence for group-type partial actions of groupoids was also developed in \cite[Theorem 5.7]{BST}, and a characterization for determined Galois extensions in terms of a partial isomorphism groupoid was given in \cite{corta}.

In the case of groupoids acting on a noncommutative ring, in \cite{pataIII} it was presented a characterization of central $\beta$-Galois algebras that have bijective Galois map, from the wide subgroupoids of $\G$ to the $R^{\beta}$-separable subalgebras of the ring $R$. Also, given a noncommutative Galois extension $R^{\beta} \subset R$, where $\beta = (\{E_g\}_{g \in \G}, \{\beta_g\}_{g \in \G})$ a unital action of the groupoid $\G$ on $R$, $C(R)$ the center of $R$ and $J_g = \{r \in E_g \mid r\beta_g(x1_{g^{-1}}) = xr,\, \text{for all}\,\ x \in R\}$, it was shown that if $J_g \neq 0$ for all $g \in \G$, then the Galois map is injective \cite[Theorem 3.3]{pataIII}. The condition $J_g \neq 0$ for all $g \in \G$ is satisfied if $R$ is a Hirata separable $\beta$-Galois extension or a central $\beta$-Galois algebra \cite[Theorem 3.5]{pataIII}. On the other hand, if the extension is commutative, then $J_g = \{0\}$ for all $g \notin \G_0$, where $\G_0$ is the set of identities of $\G$. In general, there are $\beta$-Galois extensions such that $J_h = \{0\}$ and $J_l \neq \{0\}$, for $h, l \in \G$, but $h, l \notin \G_0$. In this direction, we present in this paper more general conditions to classify $\beta$-Galois extensions with injective map. The notation of $J_g$ was inspired in the works developed by Szeto and Xue in \cite{szetoI}, \cite{szetoII}, \cite{szetoIII} and \cite{szetoIV}, in which they treated about Galois extensions for actions of groups.

The paper is organized as follows. We start with some preliminary results about groupoids, groupoid actions, Galois extensions and separability. In Section 2 we present two induced maps by the Galois map $\si:\HH \mapsto \tet(\HH)C(R)$ and $\gamma:\HH \mapsto V_R(\tet(\HH))$  and another two maps $\overline{\si}:\overline{\HH} \mapsto \tet(\HH)C(R)$ and $\overline{\ga}:\overline{\HH} \mapsto V_R(\tet(\HH))$ induced by $\sigma$ and $\gamma$, respectively. We prove that $\overline{\sigma}$ and $\overline{\gamma}$ are both injectives if $R$ is a $\beta$-Galois extension such that $R^{\beta}$ is separable over $C(R)^{\beta}$. In Section 4 we show the relation between $\si$ and $\gamma$ and we give conditions to the Galois map be injective, which generalizes results in \cite{pataIII}, \cite{szetoIII} and \cite{szetoIV}.

Throughout, unless otherwise specified, rings (hence, also algebras) are associative and unital.

\section{Preliminary}\label{sec:preliminares}

\subsection{Groupoids}
A {\it{groupoid}} $\G$ is a small category in which every morphism is invertible. Given $g\in \G$, the {\it domain} and the {\it  range }  of $g$ will be denoted by $d(g)$ and $r(g)$, respectively. Also, $\G_0$ will denote the set of the objects of $\G$. Hence, we have maps $d,r:\G \to \G_0$.  Given $e\in \G_0$, $id_e$ will denote the identity morphism of $e$.  Observe that $id:\G_0\to \G$, given by $id(e)=id_e$, is an injective map. Thus, we identify $\G_0\subseteq \G$. 

Let $\GG=\{ (g,h) \in \G \times \G : d(g)=r(h)\}$. The map $m: \GG \to \mathcal{G}$,  $m(g,h)=gh$, is called \emph{composition map}. For each pair of objects $e,f\in \mathcal{G}_0$, denote by $\mathcal{G}(e,f)=\{g \in \G: d(g)=e, r(g)=f\}$. In particular $\mathcal{G}(e,e)=\mathcal{G}_e$ is a group, called the {\it isotropy (or principal) group associated to $e$}. 

A non-empty subset $\mathcal{H}$ of $\G$ is called a {\it subgroupoid} of $\G$ if it is stable by multiplication and by inverse. If, in addition, $\mathcal{H}_{0} = \G_{0}$, then $\mathcal{H}$ is said to be a {\it wide subgroupoid}. We will use the notation $\HH < \G$.

\subsection{Actions of groupoids}
We recall some basic results about actions of groupoids. The reference that will be used here is \cite{bagio}. Consider $R$ an algebra over a commutative ring $K$. An \emph{action} \emph{of} $\G$ \emph{over} $R$ is a pair $$\beta = (\{E_g\}_{g \in \G}, \{\beta_g\}_{g \in \G})$$ where for each $g \in \G$, $E_g = E_{r(g)}$ is an ideal of $R$ and $\beta_g: E_{g^{-1}} \to E_g$ is an isomorphism of $K$-algebras satisfying the following conditions: \begin{itemize}
\item [(i)] $\beta_e$ is the identity map $Id_{E_e}$ of $E_e$ for all $e \in \G_0$;
    \item [(ii)] $\beta_g(\beta_h(r)) = \beta_{gh}(r)$ for all $(g, h) \in \GG$ and for all $r \in E_{h^{-1}} = E_{(gh)^{-1}}$.\end{itemize}
If each $E_g$, $g\in\G$,  is a unital algebra (with identity element denoted by $1_g$) we say that the action $\beta$ is \emph{unital}.

The \emph{skew groupoid ring} $R \star_{\beta} \G$, corresponding to an action $\beta$ of a groupoid $\G$ on an algebra $R$, is defined as the direct sum $$R \star_{\beta} \G = \bigoplus_{g \in \G}E_gu_g,$$ (where each $u_g$ is a placeholder for the g-th component) with the usual addition and the multiplication induced by the rule
$$
(xu_g)(yu_h) = \left\{
\begin{array}{lll}
x\beta_g(y1_{g^{-1}})u_{gh},& \mbox{if} \quad (g, h) \in \G_2\\
0, & \mbox{otherwise,}
\end{array}
\right.
$$ for all $g, h \in \G, x \in E_g$ and $y \in E_h$. It is straightforward to check that $R \star_{\beta} \G$ is associative and, if $\G_0$ is finite, also unital, with identity element given by $1_{R \star_{\beta} \G}=\sum_{e\in\G_0}1_e\delta_e$.

\subsection{Galois extension and separability}
For any action $\beta$ of  a groupoid $\G$ on an algebra $R$ we will denote by $$R^{\beta}=\{r\in R\ |\ \beta_g(r1_{g^{-1}})=r1_g,\,\ \text{for all}\,\ g\in\G\}$$ the subalgebra of $R$ of the elements which are invariant under $\beta$.

We say that $R$ is a $\beta$-\emph{Galois extension of} $R^{\beta}$ if $\G$ is finite and there exist elements $x_i, y_i \in R$, $1 \leq i \leq m$, such that $\sum_{1 \leq i \leq m}x_i\beta_g(y_i1_{g^{-1}}) = \sum_{e \in \G_0}\delta_{e, g}1_e$, for all $g \in \G$. The set $\{x_i, y_i\}_{1\leq i \leq m}$ is called a \emph{Galois coordinate system} of $R$ over $R^{\beta}$.

Let $R\supseteq S$ be a ring extension. We say that $R$ is \emph{separable over $S$} if there exists an element $z = \sum_{i = 1}^nx_i \otimes_S y_i \in R \otimes_S R$ such that $\sum_{i = 1}^nx_iy_i = 1_R$ and $rz = zr$, for all $r\in R$ \cite{HS}. If the ring $S$ is commutative, it is equivalent to say that $R$ is a projective left $R^e$-module, where  $R^e = R \otimes_S R^o$ and $R^o$ is the opposite algebra of $R$ \cite{chase1969galois}.

\section{Induced Maps}
From now on, assume that $\G$ is a finite groupoid, $K$ is a commutative ring, $R$ is a $K$-algebra, $\beta = (\{E_g\}_{g \in \G}, \{\beta_g\}_{g \in \G})$ is a unital action of $\G$ on $R$  such that $R = \bigoplus_{e \in \G_0}E_e$, $C(R)$ is the center of $R$, $V_{S_2}(S_1) = \{r \in S_2 \mid rs = sr$, for all $s \in S_1\}$ is the commutator of $S_1$ in $S_2$, for any subrings $S_1, S_2$ of $R$ and $J_g$, $g\in\G$, is the $C(R)$-submodule of $R$ defined by $J_g = \{r \in E_g \mid r\beta_g(x1_{g^{-1}}) = xr,\, \text{for all}\,\ x \in R\}$. It follows by \cite[Proposition 2.2]{pataIII} that if $\HH$ is a subgroupoid of $\G$ and $R_{\HH} = \bigoplus_{e \in \HH_0}E_e$, then $\beta_{\HH} = \{\beta_h: E_{h^{-1}} \to E_h \mid h \in \HH\}$ is an action of $\HH$ on $R_{\HH}$ and $R_{\HH}$ is a $\beta_{\HH}$-Galois extension of $(R_{\HH})^{\beta_{\HH}}$. Hence, if the groupoid is wide, we have that $\HH$ acts on $R$ and $(R_{\HH})^{\beta_{\HH}} = R^{\beta_{\HH}}$.

We start recalling a crucial result from \cite{pataIII}.

\begin{lema}\cite[Lemma 3.1]{pataIII} \label{lem1} In the same notations above, we have that
\begin{align*}
	V_R(\R)=\bigoplus_{g \in \G} J_g.
\end{align*}
\end{lema}

Next we present some useful properties about the commutator.

\begin{prop}\label{prop1}
	Assume $S_1, S_2$ subrings of $R$. Then:
	\begin{itemize}
		\item [(i)] $V_R(S_1) = V_R(S_1C(R))$;
		\item [(ii)] If $V_R(S_1)=V_R(S_2)$, then $V_R(S_1C(R))=V_R(S_2C(R))$;
		\item [(iii)] If $S_1 \subseteq C(R)$, then $V_R(S_1)=R$.
	\end{itemize}
\end{prop}
\begin{proof}
	$(i)$ Consider $r \in V_R(S_1)$. Then $rs=sr$, for all $s \in S_1$. Given $s'c \in S_1C(R)$, 
	\begin{align*}
	r(s'c)=r(cs')=(rc)s'=(cr)s'=c(rs')=c(s'r)=(cs')r=(s'c)r.
	\end{align*}
	Therefore, $r \in V_R(S_1 C(R))$. The reciprocal is immediate, since $S_1 \subseteq S_1C(R)$.

	$(ii)$ Let $r \in V_R(S_1C(R))$. Then $r(sc)=(sc)r$, where $sc \in S_1C(R)$. Notice that $r(sc)=r(cs)=(cr)s$ and $(sc)r=s(cr)$, once $c \in C(R)$. Thus $(cr)s=s(cr)$, whence $cr \in V_R(S_1)$. Since $V_R(S_1)=V_R(S_2)$, then $cr \in V_R(S_2)$, that is, $(cr)s'=s'(cr)$ for all $s' \in S_2$. Hence $(cr)s'=(rc)s'=rcs'=r(s'c)$ and $s'(cr)=(s'c)r$. Therefore $r \in V_R(S_1C(R))$. The reciprocal is analogous. 
	
	$(iii)$ It is immediate.
\end{proof}

For the rest of the paper, given a wide subgroupoid $\HH$ of $\G$, we will fix the notation $\tet:\HH \mapsto R^{\bet_{\HH}}$ for the Galois map from the set of the wide subgroupoids of $\G$ to the set of $R^{\beta}$-separable algebras of $R$. 

Now suppose that $R$ is a $\bet$-Galois extension of $R^{\bet}$ such that $R^{\bet}$ is separable over $C(R)^{\bet}$. For any subgroupoid $\HH$ of $\G$, let $S_\HH=\{g \in \HH \, | \, J_g \neq \{0\}\}$, $T_\HH=\{g \in \HH \, | \, J_g = \{0\}\}$ and $\overline{\HH}=\{\L \, | \, \L < \G \,  \, \, \textnormal{and} \, \, S_\L=S_\HH\}$. Define two maps induced by $\tet$: $\si:\HH \mapsto \tet(\HH)C(R) $ and $\gamma:\HH \mapsto V_R(\tet(\HH))$.

We want to show that $\overline{\si}:\overline{\HH} \mapsto \tet(\HH)C(R)$ and $\overline{\ga}:\overline{\HH} \mapsto V_R(\tet(\HH))$ are injective maps from $\mathcal{A}=\{\ \overline{\HH} \, | \, \HH < \G \text{ and } \HH_0 = \G_0 \}$ to the set $\mathcal{B}$ of the separable subalgebras of $R$ over $C(R)$. For this, we need first to prove some results.

\begin{lema}\label{lem2}
	Let $R$ be a $\bet$-Galois extension of $R^{\bet}$. If $\R$ is separable over $\cbet$, then $R \star_{\bet} \G$ is also separable over $\cbet$.
\end{lema}
\begin{proof}
	Since $R$ is a $\bet$-Galois extension of $R^{\bet}$, by \cite[Theorem 5.3]{bagio}, $R$ is a finitely generated projective right $\R$- module and $End_{\R}(R,R) \simeq R \star_{\bet} \G$. Thus, by \cite[Theorem 1]{kan}, $End_\R(R,R)$ is a separable algebra over $\cbet$. Consequently, $R \star_{\bet} \G$ is separable over $\cbet$.
\end{proof}

\begin{lema}\label{lem3}
	Let $R$ be a $\bet$-Galois extension of $R^{\bet}$  such that $R^{\bet}$ is separable over $C(R)^{\bet}$. Then $\Rh$ is separable over $\cbet$, for each wide subgroupoid $\HH$ of $\G$, where $\bet_\HH=\{\bet_h:E_{h^{-1}} \seta E_h \, | \, h \in \HH\}$ is an action of $\HH$ on $R$.
\end{lema}
\begin{proof}
Since $R^{\bet}$ is separable over $C(R)^{\bet}$, by \autoref{lem2}, $R \star_{\bet} \G$ is separable over $\cbet$. Observe that $\G=\dot{\cup} \, g\HH$, where $g\HH=\{gh \, | \, h \in \HH \, \textnormal{and} \, d(g)=r(h)\}$. Therefore, $\G=g_1\HH + \cdots + g_n\HH=\HH g'_1 + \cdots + \HH g'_n$ with $g_1=d(g_1)$ and $g'_1=r(g'_1)$.
	\\
	\textbf{Step 1:} $R \star_{\bet} \G=R \star_{\bet_\HH}\HH \oplus \sum_{i=2}^n (R \star_{\bet_\HH} \HH)u_{g'_i} = \bigoplus_{i=1}^n (R \star_{\bet_\HH} \HH)u_{g'_i}$.
	\\
	($\subseteq$) Given $g \in \G$, there exists $h \in \HH$ such that  $g=hg'_i$ with $d(h)=r(g'_i)$ and some $1 \leq i \leq n$. Thus, given $\sum_{g \in \G}x_gu_g \in R \star_{\bet} \G$,
	\begin{align*}
	\sum_{g \in \G}x_gu_g &=\sum_{i=1}^n\sum_{\substack{h \in \HH, \\ d(h) = r(g'_i)}} x_{hg'_i}u_{hg'_i}=\sum_{\substack{h \in \HH, \\ d(h) = r(g'_1)}}x_{hg'_1}u_{hg'_1}+ \sum_{i=2}^n\sum_{\substack{h \in \HH, \\ d(h) = r(g'_i)}} x_{hg'_i}u_{hg'_i} \\
	&=\sum_{\substack{h \in \HH, \\ d(h) = r(g'_i)}}x_{hr(g'_1)}u_{hr(g'_1)}+\sum_{i=2}^n\sum_{\substack{h \in \HH, \\ d(h) = r(g'_i)}} x_{hg'_i}u_{hg'_i} \\
	&=\sum_{\substack{h \in \HH, \\ d(h) = r(g'_i)}}x_{hd(h)}u_{hd(h)}+\sum_{i=2}^n\sum_{\substack{h \in \HH, \\ d(h) = r(g'_i)}} x_{hg'_i}u_{hg'_i} \\
	&=\sum_{\substack{h \in \HH, \\ d(h) = r(g'_i)}}x_{h}u_{h}+\sum_{i=2}^n\sum_{\substack{h \in \HH, \\ d(h) = r(g'_i)}} x_{hg'_i}u_{h}u_{g'_i} \in R \star_{\bet_\HH}\HH + \sum_{i=2}^n (R \star_{\bet_\HH} \HH)u_{g'_i}.
	\end{align*}
	\\
	($\supseteq$) Since $R \star_{\bet_\HH} \HH \subseteq R \star_{\bet} \G$, it follows that $R \star_{\bet_\HH}\HH \oplus \sum_{i=2}^n (R \star_{\bet_\HH} \HH)u_{g'_i} \subseteq R \star_{\bet}\G$.
	
	The last equality of the claim is immediate. Furthermore, it is not difficult to see that $\sum_{i=1}^n(R \star_{\bet_\HH} \HH)u_{g'_i}=\sum_{i=1}^nu_{g_i}(R \star_{\bet_\HH} \HH).$ Then $R \star_{\bet} \G=R \star_{\bet_\HH}\HH \oplus \sum_{i=2}^n (R \star_{\bet_\HH} \HH)u_{g'_i}=R \star_{\bet_\HH}\HH \oplus \sum_{i=2}^n u_{g_i}(R \star_{\bet_\HH} \HH)$, and, consequently, $\bigoplus_{i=1}^n(R \star_{\bet_\HH}\HH)u_{g'_i}=\bigoplus_{i=1}^nu_{g_i}(R \star_{\bet_\HH}\HH)$. 
	\\
	\textbf{Step 2:} $R \star_{\bet_\HH}\HH$ is separable over $\cbet$.
	\\
	Clearly, $R \star_{\bet_\HH}\HH$ is a subring of $R \star_{\bet}\G$. Once $R \star_{\bet}\G$ is separable over $C(R)^\bet$, $R \star_{\bet}\G$ is a projective left $(R \star_{\bet}\G)^e$- module. 
	Moreover,
	\begin{align*}
	(R \star_{\bet}\G)^e &= R \star_{\bet}\G \otimes_{C(R)^\bet} (R \star_{\bet}\G)^o  \\ 
	&=  \bigoplus_{i=1}^n (R \star_{\bet_\HH}\HH)u_{g'_i} \otimes_{C(R)^\bet} (\bigoplus_{i=1}^n (R \star_{\bet_\HH}\HH)u_{g'_i})^o  \\
	&=  \bigoplus_{i=1}^n (R \star_{\bet_\HH}\HH)u_{g'_i} \otimes_{C(R)^\bet} \bigoplus_{i=1}^n (R \star_{\bet_\HH}\HH)^o u_{g'_i}^o \\
		&= \bigoplus_{i=1}^n ((R \star_{\bet_\HH}\HH) \otimes_{C(R)^\bet} (R \star_{\bet_\HH}\HH)^{o})(u_{g'_i} \otimes u_{g'_i}^o) \\
	&= \bigoplus_{i=1}^n (R \star_{\bet_\HH}\HH)^{e}(u_{g'_i} \otimes u_{g'_i}^o).
	\end{align*}
	Thus, $R \star_{\bet}\G$ is a projective left $(R \star_{\bet_\HH}\HH)^{e}$- module. Furthermore, since $R \star_{\bet_\HH}\HH$ is a direct summand of $R \star_{\bet}\G$ as a left $(R \star_{\bet_\HH}\HH)^{e}$-module, we have that $R \star_{\bet_\HH}\HH$ is a projective left $(R \star_{\bet_\HH}\HH)^{e}$- module. Thus, $R \star_{\bet_\HH}\HH$ is separable over $C(R)^{\beta}$.

	Since $R$ is a $\bet$-Galois extension of $R^{\bet}$, by \cite[Theorem 5.3]{bagio} $R$ is a finitely generated right $\R$-module and by $\R$ being separable over $\cbet$, by \cite[Theorem 1]{kan} we have that $R$ is a finitely generated right $End_{\R}(R,R) \simeq R \star_{\bet} \G$- module. Notice that $R \star_{\bet} \G$ is a free $R \star_{\bet_\HH} \HH$-module, because $R \star_{\bet_\HH} \HH \subseteq R \star_{\bet} \G$ and $1_{R \star_{\bet_\HH} \HH}=1_{R \star_{\bet} \G}$. Thus $R \star_{\bet} \G$ is a finitely generated projective $R \star_{\bet} \HH$-module. Then $R$ is a finitely generated projective $R \star_{\bet} \HH$-module. Thereby, by Step 2 and by \cite[Theorem 1]{kan}, $End_{R \star_{\bet_\HH}\HH}(R,R)$ is separable over $\cbet$. Once $End_{R \star_{\bet_\HH}\HH}(R,R) \simeq R^{\bet_\HH}$ as rings, we have that $R^{\bet_\HH}$ is separable over $\cbet$.
\end{proof}

\begin{lema}\label{lem4}
	Let $R$ be a $\bet$-Galois extension of $R^{\bet}$ such that $R^{\bet}$ is separable over $C(R)^{\bet}$. Then $\gaa:\overline{\HH} \seta V_R(\tet(\HH))$, for $\HH$ subgroupoid of $\G$, is well-defined. 
\end{lema}
\begin{proof}
	Take $\L \in \overline{\HH}$. We are going to prove that $V_R(\tet(\L))=V_R(\tet(\HH))$. 
	
	Since $\L \in \overline{\HH}$, $S_\L=S_\HH$. Then, $\bigoplus_{h \in S_\HH}J_h=\bigoplus_{g \in S_\L}J_g$. But $\bigoplus_{h \in \HH}J_h=\bigoplus_{h \in S_\HH}J_h$ and $\bigoplus_{g \in \L}J_g=\bigoplus_{g \in S_\L}J_g$. Therefore, $\bigoplus_{h \in \HH}J_h=\bigoplus_{g \in \L}J_g$ and, consequently, by \autoref{lem1}, $V_R(R^{\bet_\L})=V_R(R^{\bet_\HH})$. Thereby, $V_R(\tet(\L))=V_R(\tet(\HH))$.
\end{proof}

\begin{lema} \label{lem5}
	Let $R$ be a $\bet$-Galois extension of $R^{\bet}$ such that $R^{\bet}$ is separable over $C(R)^{\bet}$. Then $\sii:\overline{\HH} \seta \tet(\HH)C(R)$, for $\HH$ wide subgroupoid of $\G$, is well-defined. 
\end{lema}
\begin{proof}
	Take $\L \in \overline{\HH}$. We are going to prove that $\tet(\HH)C(R)=\tet(\L)C(R)$. Since $\L \in \overline{\HH}$, by \autoref{lem4}, $V_R(\tet(\L))=V_R(\tet(\HH))$, and by \autoref{prop1}$(i)$, $V_R(\tet(\L)C(R))\!=\!V_R(\tet(\HH)C(R))$. From \autoref{lem3}, $R^{\bet_\L}=\tet(\L)$ and $R^{\bet_\HH}=\tet(\HH)$ are both separable over $C(R)^\bet$. 
	\\
	\textbf{Claim:} $\tet(\HH)C(R)$ and $\tet(\L)C(R)$ are both separable over $C(R)$.
	\\
	We are going to check that $\tet(\HH)C(R)$ is separable over $C(R)$. Since $\tet(\HH)$ is separable over $\cbet$, there exists $e=\sum_{i=1}^n x_i \otimes y_i \in \tet(\HH) \otimes_{\cbet} \tet(\HH)$ such that $\sum_{i=1}^n x_iy_i=1_R \, \textnormal{and} \, \sum_{i=1}^n rx_i \otimes y_i=\sum_{i=1}^n x_i \otimes ry_i$ for all $r \in R^{\bet_\HH}$. Take $d=\sum_{i=1}^n x_i1_R \otimes_{C(R)} y_i1_R \in \tet(\HH) C(R) \otimes_{C(R)} \tet(\HH) C(R)$. Thus, 
	\begin{align*}
	\sum_{i=1}^n (x_i1_R)(y_i1_R)=\sum_{i=1}^n(x_iy_i)1_R=\sum_{i=1}^n x_iy_i=1_R.
	\end{align*}
	Now, note that $\tet(\HH) C(R) \otimes_{C(R)} \tet(\HH) C(R)$ is a left $C(R)$-module via $c \cdot (xa \otimes_{C(R)} yb)=c(xa) \otimes_{C(R)} yb=x(ca) \otimes_{C(R)} yb$, for any  $x,y \in \tet(\HH)$ and $c, a, b \in C(R)$. Thus, given $xc \in \tet(\HH) C(R)$
	\begin{align*}
	\sum_{i=1}^n (xc)x_i1_R \otimes_{C(R)} y_i1_R&= \sum_{i=1}^n c(xx_i)1_R \otimes_{C(R)} y_i1_R \\
	&=(c \otimes_{C(R)} 1_R) \sum_{i=1}^n xx_i1_R \otimes_{C(R)} y_i1_R \\
	&= (c \otimes_{C(R)} 1_R) \sum_{i=1}^n x_i1_R \otimes_{C(R)} xy_i1_R \\
	&= \sum_{i=1}^n cx_i1_R \otimes_{C(R)} xy_i1_R \\
	&= \sum_{i=1}^n x_i1_R \otimes_{C(R)} c(xy_i1_R) \\
	&=\sum_{i=1}^n x_i1_R \otimes_{C(R)} (xc)y_i1_R.
	\end{align*}
	Thereby, $\tet(\HH)C(R)$ is separable over $C(R)$. Analogously, $\tet(\L)C(R)$ is separable over $C(R)$.
	
	Since $R$ is a $\bet$-Galois extension of $\R$, $R$ is a separable extension of $\R$. Moreover, $\R \supseteq \cbet$. Once $\R$ is a separable extension of $\cbet$, it follows by \cite[Proposition 2.5 (2)]{HS} that $R$ is a separable extension of $\cbet$. Also, we have that $C(R) \supseteq \cbet$. Thus, by \cite[Proposition 2.5 (1)]{HS}, $R$ is a separable extension of $C(R)$, that is, $R$ is an Azumaya algebra. Clearly, $\tet(\HH)C(R)$ and $\tet(\L)C(R)$ are both subalgebras of $R$. Therefore, by \cite[Theorem 4.3]{demeyer}, $\tet(\HH)C(R)=V_R(V_R(\tet(\HH)C(R)))$ and $\tet(\L)C(R)=V_R(V_R(\tet(\L)C(R)))$. Thus, 
	\begin{align*}
	\tet(\HH)C(R)=V_R(V_R(\tet(\HH)C(R)))=V_R(V_R(\tet(\L)C(R)))=\tet(\L)C(R).
	\end{align*}
\end{proof}

\begin{lema}\label{lem6}
	Let $R$ be a $\bet$-Galois extension of $R^{\bet}$ and $\phi:S \mapsto \bigoplus_{g \in S}J_g$ for $S \subseteq S_\G$. Then $\phi$ is an injectivy map.
\end{lema}
\begin{proof}
	First of all, let us check that $\phi$ restricted to the set of unitaries subsets of $S_\G$ is injective. Let $g, h \in S_\G$ such that $\phi(\{g\})=\phi(\{h\})$, that is, $J_g=J_h$. Consequently, $E_g \cap E_h \neq \{0\}$ and then $r(g)=r(h)$. Since $J_g=J_h \neq \{0\}$, there exists $0 \neq b \in E_g=E_h$ such that $b \in J_g=J_h$ and 
	\begin{align*}
	xb=b\bet_g(x1_{g^{-1}})=b\bet_h(x1_{h^{-1}})
	\end{align*}
	for all $x \in R$. Thus,  $b(\bet_g(x1_{g^{-1}})-\bet_h(x1_{h^{-1}}))=0$ for all $x \in R$. Putting $\bet_g$ in evidence, $b\bet_g((x1_{d(g)})-\bet_{g^{-1}h}(x1_{h^{-1}g}))=0$ for all $x \in R$. Notice that $\exists \, g^{-1}h$, since $r(g)=r(h)$. Moreover, since $b \in J_g$, $b\bet_g(x1_{g^{-1}}) = xb$ for all $x \in R$. Hence,
	\begin{align}\label{eq1}
	b\bet_g(x1_{d(g)}-\bet_{g^{-1}h}(x1_{h^{-1}g}))=(x1_{d(g)}-\bet_{g^{-1}h}(x1_{h^{-1}g})b=0
	\end{align}
	for all $x \in R$. Since $R$ is a $\bet$-Galois extension of $R^{\bet}$, there exist $x_i, y_i \in R, 1 \leq i \leq n$ such that  $\sum_{i=1}^n x_i\beta_g(y_i1_{g^{-1}})= \sum_{e \in \G_0}\delta_{e,g}1_e$, for all $g \in \G$ and $e \in \G_0$. Changing $x$ by $y_i$ and multiplying by $x_i$ in \eqref{eq1} we have that $$x_i(y_i1_{d(g)}-\beta_{g^{-1}h}(y_i1_{h^{-1}g}))b=0$$ for all $1 \leq i \leq n$. Then,
	\begin{align*}
	\sum_{i=1}^nx_i(y_i1_{d(g)}-\bet_{g^{-1}h}(y_i1_{h^{-1}g}))b = 0 \Rightarrow (\sum_{i=1}^n x_iy_i)b=(\sum_{i=1}^n x_i\bet_{g^{-1}h}(y_i1_{h^{-1}g}))b.
	\end{align*}
	Thereby,
	\begin{align*}
	b&=1_{r(g)}b=(\sum_{i=1}^n x_i\bet_{r(g)}(y_i1_{r(g)}))b=(\sum_{i=1}^nx_iy_i1_{r(g)})b=(\sum_{i=1}^n x_iy_i)b \\
	&=(\sum_{i=1}^n x_i\bet_{g^{-1}h}(y_i1_{h^{-1}g}))b=(\sum_{e \in \G_0}\delta_{e,{g^{-1}h}}1_e)b.
	\end{align*}
	Since $b \neq 0$ we have that $g^{-1}h \in \G_0$ and then
	\begin{align*}
	g^{-1}h=r(g^{-1}h)=r(g^{-1})=d(g).
	\end{align*}
	Thus,
	\begin{align*}
	h=r(h)h=r(g)h=gg^{-1}h=gd(g)=g.
	\end{align*}
	Therefore, $\{g\}=\{h\}$. Finally, let $S, S'$ be two non-empty subsets of $S_\G$ such that $\phi(S)=\phi(S')$. That is, $\bigoplus_{g \in S}J_g=\bigoplus_{h \in S'}J_h$. 
	\\
	\textbf{Claim:} For each $l \in S$ we have that $l \in S'$.
	\\
	Suppose that $l \notin S'$. Since 
	\begin{align*}
	\bigoplus_{g \in S_\G}J_g=(\bigoplus_{h \in S'}J_h) \oplus (\bigoplus_{h \notin S'}J_h )
	\end{align*}
	by the first part of the proof, $J_l \bigcap (\bigoplus_{h \in S'}J_h) =\{0\} $ for $l \notin S'$. Thus, $J_l \nsubseteq \bigoplus_{h \in S'}J_h$, which is a contradiction, because $J_l \subseteq \bigoplus_{g \in S}J_g=\bigoplus_{h \in S'}J_h$, for $l \in S$. Then $l \in S'$ and therefore $S \subseteq S'$. 
	
	Similarly, $S' \subseteq S$. Hence $S=S'$.
\end{proof}

\begin{teo}\label{teo2}
	Let $R$ be a $\bet$-Galois extension of $R^{\bet}$ such that $R^{\bet}$ is separable over $C(R)^{\bet}$. Then $\sii:\overline{\HH} \mapsto \tet(\HH)C(R)$ and $\gaa:\overline{\HH} \mapsto V_R(\tet(\HH))$, $\overline{\HH} \in \mathcal{A}$, are both injective maps.
\end{teo}
\begin{proof}
	Suposse that $\sii(\overline{\HH})=\sii(\overline{\L})$. Note that 
	\begin{align*}
	R^{\bet_\HH} C(R)=\tet(\HH)C(R)=\sii(\overline{\HH})=\sii(\overline{\L})=\tet(\L)C(R)=R^{\bet_\L}C(R)
	\end{align*}
	for $\HH, \L$ subgroupoids of $\G$. Thus, $V_R(R^{\bet_\HH} C(R))=V_R(R^{\bet_\L}C(R))$. By \autoref{prop1} $(ii)$, $V_R(R^{\bet_\HH})=V_R(R^{\bet_\HH} C(R))$ and $V_R(R^{\bet_\L})=V_R(R^{\bet_\L}C(R))$.
	
	Then $V_R(R^{\bet_\HH})=V_R(R^{\bet_\HH} C(R))=V_R(R^{\bet_\L}C(R))=V_R(R^{\bet_\L})$. By \autoref{lem1}, $\bigoplus_{h \in S_\HH}J_h=\bigoplus_{l \in S_\L}J_l$ and by \autoref{lem6} $S_\HH=S_\L$. So $\overline{\HH}=\overline{\L}$. Therefore, $\sii$ is injective. 
	
	Now, suppose that $\gaa(\overline{\HH})=\gaa(\overline{\L})$. Then, $V_R(\tet(\HH))=V_R(\tet(\L))$. By \autoref{lem1}, 
	\begin{align*}
	\bigoplus_{h \in S_\HH}J_h=V_R(R^{\bet_\HH})=V_R(\tet(\HH))=V_R(\tet(\L))=V_R(R^{\bet_\L})=\bigoplus_{l \in S_\L}J_l.
	\end{align*}
	From \autoref{lem6}, $S_\HH=S_\L$ and then $\overline{\HH}=\overline{\L}$. Therefore, $\gaa$ is injective.
\end{proof}

\section{Injectivity of the Galois map}

In this section we will see conditions under which the Galois map $\tet:\HH \mapsto R^{\bet_{\HH}}$ is injective.

\begin{lema}\label{lem7} The map $\si:\HH \mapsto \tet(\HH)C(R)$ is injective if and only if the map $\ga:\HH \mapsto V_R(\tet(\HH))$ is injective.
\end{lema}
\begin{proof}
	Suppose that $\si$ is injective. Let $\HH, \L$ be subgroupoids of $\G$ such that $\ga(\HH)=\ga(\L)$. Then $V_R(\tet(\HH))=V_R(\tet(\L))$, that is, $V_R(R^{\bet_{\HH}})=V_R(R^{\bet_{\L}})$. Thus, by \autoref{prop1} $(ii)$, $V_R(R^{\bet_{\HH}} C(R))=V_R(R^{\bet_{\L}} C(R))$. Since $R^{\bet_{\HH}}$ and $R^{\bet_{\L}}$ are both separable algebras over $\cbet$ by \autoref{lem3}, then $R^{\bet_{\HH}} C(R)$ and $R^{\bet_{\L}} C(R)$ are separable over $C(R)$. Moreover, $R$ is an Azumaya algebra. Then, by \cite[Theorem 4.3]{demeyer},
	\begin{align*}
	R^{\bet_{\HH}} C(R)=V_R(V_R(R^{\bet_{\HH}} C(R)))=V_R(V_R(R^{\bet_{\L}} C(R)))=R^{\bet_{\L}} C(R).
	\end{align*}
	Thereby $\si(\HH)=\si(\L)$. Since $\si$ is injective, it follows that $\HH=\L$. Therefore, $\ga$ is injective. Reciprocally, supose that $\ga$ is injective. Let $\HH, \L$ be subgroupoids of $\G$ such that $\si(\HH)=\si(\L)$. Then $\tet(\HH)C(R)=\tet(\L)C(R)$. From \autoref{prop1} $(i)$, $V_R(\tet(\HH))=V_R(\tet(\HH)C(R))$ and $V_R(\tet(\L))=V_R(\tet(\L)C(R))$. Thus 
	\begin{align*}
	V_R(\tet(\HH))=V_R(\tet(\HH)C(R))=V_R(\tet(\L))=V_R(\tet(\L)C(R)).
	\end{align*}
	So $\ga(\HH)=\ga(\L)$. Since $\ga$ is injective, we have that $\HH=\L$. Therefore, $\si$ is injective.
\end{proof}

\begin{lema}\label{lem8} If either the map $\si:\HH \mapsto \tet(\HH)C(R)$ or $\ga:\HH \mapsto V_R(\tet(\HH))$ is injective, then the Galois map $\tet:\HH \mapsto R^{\bet_{\HH}}$ is injective.
\end{lema}
\begin{proof}
	Assume that $\si$ is injective. Let $\HH, \L$ be subgroupoids of $\G$ such that $\tet(\HH)=\tet(\L)$. Then, 
	\begin{align*}
	R^{\bet_{\HH}}=R^{\bet_{\L}} \Rightarrow R^{\bet_{\HH}} C(R)= R^{\bet_{\L}} C(R) \Rightarrow \si(\HH)=\si(\L).
	\end{align*}
	Since $\si$ is injective, then $\HH=\L$. Thus, $\tet$ is injective. 
	
	Now supposing that $\ga$ is injective, by \autoref{lem7}, $\si$ is injective. Consequently, $\tet$ is injective. 
\end{proof}

\begin{teo}\label{teo3} If $\overline{\HH}=\{\HH\}$, that is, a singleton, for each $\HH$ wide subgroupoid of $G$, then $\tet$ is injective.
\end{teo}
\begin{proof}
	Let $\HH, \L$ be subgroupoids of $\G$. Then, 
	\begin{align*}
	R^{\bet_{\HH}} = R^{\bet_{\L}} \Rightarrow R^{\bet_{\HH}} C(R) = R^{\bet_{\L}} C(R).
	\end{align*}
	Thus, by \autoref{lem5}, $\si(\overline{\HH})=\si(\overline{\L})$. As $\si$ is injective by \autoref{teo2}, it follows that $\overline{\HH}=\overline{\L}$. Since $\overline{\HH}=\{\HH\}$ and $\overline{\L}=\{\L\}$, then $\HH=\L$. Thus $\tet$ is injective.
\end{proof}

Next theorem extends \cite[Theorem 3.3]{pataIII}.

\begin{teo}\label{teo4} Let $\langle S_\HH \rangle$ be the subgroupoid of $\G$ generated by the elements of $S_\HH$, for $\HH$ wide subgroupoid of $\G$. If $\langle S_\HH\rangle =\HH$ for each wide subgroupoid $\HH$ of $\G$, then $\tet$ is injective.
\end{teo}
\begin{proof}
	Let $\HH, \L$ be subgroupoids of $\G$ such that $\tet(\HH)=\tet(\L)$. Then $R^{\bet_{\HH}} = R^{\bet_{\L}}$. By \autoref{lem1}, $\bigoplus_{h \in \HH} J_h=V_R(R^{\bet_{\HH}})$ and $\bigoplus_{l \in \L} J_l=V_R(R^{\bet_{\L}})$. Thus, $\bigoplus_{h \in \HH} J_h=\bigoplus_{l \in \L} J_l$. By  \autoref{lem6}, $S_\HH\!=\!S_\L$. Consequently, $\langle S_\HH\rangle= \langle S_\L\rangle$. By hypothesis, $\HH=\L$. Therefore, $\tet$ is injective.
\end{proof}

Observe that \autoref{teo4} holds for any $\bet$-Galois extension $R$ over $\R$,  this extension not being necessarily separable over $\cbet$.

\begin{cor}\label{cor1} \cite[Theorem 3.3]{pataIII}
	If $J_g \neq \{0\}$ for each $g \in \G$, then $\tet$ is injective.
\end{cor}
\begin{proof}
	If $J_g \neq \{0\}$ for each $g \in \G$, it occurs, in particular, for each $h \in \HH$, for all $\HH$ wide subgroupoid of $\G$. Hence $\langle S_\HH \rangle \, = S_\HH=\HH$. Thus, by \autoref{teo4}, $\tet$ is injective.
\end{proof}

\bibliographystyle{amsalpha}

\begin{thebibliography}{99}

\bibitem{bagio} D. Bagio; A. Paques, \textit{Partial groupoid actions: globalization, Morita theory and Galois theory}, Comm. Algebra 40 (10) (2012), 3658-3678.

\bibitem{BST} D. Bagio; A. Sant'ana; T. Tamusiunas, \textit{Galois correspondence for group-type partial actions of groupoids}, Bull. Belgian Math. Soc. 5 (2022), 745-767.

\bibitem{chase1969galois} S. U. Chase; D. K. Harrison; A. Rosenberg, \textit{Galois theory and cohomology of commutative rings}, Am. Math. Soc. 52 (1965).

\bibitem{corta} W. Cortes; T. Tamusiunas, \textit{A characterisation for groupoid Galois extension using partial isomorphisms}, Bull. Aust. Math. Soc. 96 (1) (2017), 59-68.

\bibitem{demeyer} F. Demeyer; E. Ingraham, \textit{Separable Algebras over Commutative Rings}, Springer Verlag, Berlin, Heidelberg, New York, LNM vol 181 (1971).

\bibitem{HS} K. Hirata and K. Sugano, \textit{On semisimple extensions and separable extensions over non-commutative rings}, J. Math. Soc. Japan 18. (1966), 360-373

\bibitem{kan} T. Kanzaki, \textit{On commutor rings and Galois theory of separable algebras}, Osaka J. Math., 1 (1) (1964), 103-115, DOI: ojm/1200691004.

\bibitem{pata} A. Paques; T. Tamusiunas, \textit{A Galois-Grothendieck-type correspondence}, Algebra Discrete Math. 17 (1) (2014), 80 - 97.

\bibitem{pataII} A. Paques; T. Tamusiunas, \textit{The Galois Correspondence Theorem for Groupoid Actions}, J. Algebra, 509 (2018), 105-123.

\bibitem{pataIII} A. Paques; T. Tamusiunas, \textit{On the Galois Map for Groupoid Actions}, Comm. Algebra, 49 (3) (2021), 11037-1047.

\bibitem{szetoI} G. Szeto; L. Xue, \textit{The structure of Galois Algebras}, J. Algebra, 237(1) (2007), 238-246.

\bibitem{szetoII} G. Szeto; L. Xue, \textit{On Galois Algebras Satisfying the Fundamental Theorem}, Comm. Algebra, 35(12) (2007), 3979-3985.

\bibitem{szetoIII} G. Szeto; L. Xue, \textit{On Galois Extensions with a One-to-One Galois Map}, International J. Algebra, 5 (17) (2011), 801-807.

\bibitem{szetoIV} G. Szeto; L. Xue, \textit{The Galois map and its induced maps}, Contemp. Ring Theory 2011, (2012), 10-15.

\end{thebibliography}
{}

\end{document}